\documentclass{article}

\usepackage{amssymb,amsmath,amsthm}
\usepackage{latexsym}

\def\er{{\mathbb R}}

\def\M{{\mathcal M}}
\def\F{{\mathcal F}}
\def\G{{\mathcal G}}
\DeclareMathOperator{\spa}{span}

\def\f{{\varphi}}
\def\R{{\er}}
\def\mi{{\mu}}
\def\ra{{\rightarrow}}
\def\is#1#2{\left\langle #1 , #2 \right\rangle}

\def\eps{\varepsilon}
\newcommand{\cov}{{\rm Cov}}
\def\Epsilon{\mathcal{E}}
\newcommand{\RO}{\R_{+}}

\def\1{{\mathbf{1}}}

\newcommand{\bh}{{\bar{h}}}

\def\n2n{{\nu_2^n}}
\newcommand{\tf}{{\tilde{f}}}

\def\np1{{\nu_p}}

\def\cppb2{{C_5}}

\newcommand{\E}{{\mathbb{E}}}
\newcommand{\sleq}{{\leq}}
\DeclareMathOperator{\Int}{Int}

\DeclareMathOperator{\supp}{supp}

\newtheorem{thm}{Theorem}[section]

\newtheorem{lemma}[thm]{Lemma}
\newtheorem{cor}[thm]{Corollary}

\newtheorem{Def}[thm]{Definition} %%\numberwithin{equation}{subsection}

\title{A simpler proof of the negative association property for absolute values of measures tied to generalized Orlicz balls
}
\author{Jakub Onufry Wojtaszczyk\thanks{Partially supported by MEiN Grant no 1 PO3A 012 29}\thanks{Partially supported by the Polish Foundation for Science} (onufry@duch.mimuw.edu.pl)\\ 
Department of Mathematics, Computer Science and Mechanics \\ 
University of Warsaw, ul. Banacha 2, 02--097 Warsaw, Poland\\}
\date{}
\begin{document}

\maketitle

\begin{abstract}
Negative association for a family of random variables $(X_i)$ means that for any coordinate--wise increasing functions $f,g$ we have $$\E f(X_{i_1},\ldots,X_{i_k}) g(X_{j_1},\ldots,X_{j_l}) \leq \E f(X_{i_1},\ldots,X_{i_k}) \E g(X_{j_1},\ldots,X_{j_l})$$ for any disjoint sets of indices $(i_m)$, $(j_n)$. It is a way to indicate the negative correlation in a family of random variables. It was first introduced in 1980s in statistics, and brought to convex geometry in 2005 to prove the Central Limit Theorem for Orlicz balls.

The paper gives a relatively simple proof of negative association of absolute values for a wide class of measures tied to generalized Orlicz balls, including the uniform measures on generalized Orlicz balls.
\end{abstract}

\section{Introduction}
We shall prove a property called the {\em negative assoctiation of absolute values} for a class of measures stemming from generalized Orlicz balls. The most important case, that is the case of uniform measures on generalized Orlicz balls was considered in \cite{malcin}. The proof given there, however, was complex and difficult to understand. The more general case, proved here, could probably also be tackled using the techniques from \cite{malcin}, but the paper would likely be even harder to read. The argument in this paper, using a technique similar to the Kannan--Lovazs--Simonovits localization lemma, is much simpler. The result itself has quite a few consequences, see eg. \cite{malcin} or \cite{fleury}, we will not explore them in this note. Negative association is defined as follows:
\begin{Def} \label{NegAssDef} We say a sequence $X_1, \ldots, X_n$ of random variables is {\em negatively associated}, if for any bounded coordinate--wise increasing functions $f : \R^k \ra \R$ and $g : \R^l \ra \R$ and disjoint subsets $\{i_1,\ldots,i_k\}$ and $\{j_1,\ldots,j_l\}$ of  $\{1,2,\ldots,n\}$ we have
\begin{equation} \label{NegAss} \cov \big(f(X_{i_1},\ldots,X_{i_k}), g(X_{j_1},\ldots,X_{j_l})\big) \leq 0.\end{equation}
\end{Def}

This definition was introduced in the 1980s by Alam, Joag--Dev, Proschan and Saxena for applications in statistics. 

When in a linear space with a fixed basis $(e_1,\ldots,e_n)$, by $x_i$ we denote $\is{x}{e_i}$ for a given vector $x$. We shall write $x \leq y$ for vectors $x,y \in V$ if $x_i \leq y_i$ for all $i \in \{1,2,\ldots,\dim V\}$. By $V_+$ we denote the set $\{x \in V : 0 \leq x\}$. By $\Int_l K$ we will denote the relative interior of $K$ in $l$.

Recall the following definitions:
\begin{Def} A {\em Young function} is an increasing convex function $f: \RO \ra \RO \cup \infty$ with $f(0) = 0$ and satisfying $f(x) \neq 0$, $f(y) \neq \infty$ for some $x,y > 0$. A generalized Orlicz ball is a set in $\R^n$ given by the inequality $\sum_{i=1}^n f_i(|x_i|) \leq n$ for some Young functions $f_1,\ldots,f_n$.\end{Def}

As noted in \cite{malcin}, if $X$ is a random vector equidistributed on a 1--symmetric convex body, one should consider the negative association property not for the sequence $(X_i)$, but rather for the absolute values $(|X_i|)$. For 1--symmetric bodies this is equivalent to considering random vectors equidistributed on the positive generalized quadrant of the body (that is vectors conditioned by $X_i \geq 0$ for all $i$). Thus we shall work only on $\RO^n$ instead of $\R^n$.

Also note that the property of being an Orlicz ball is dependent upon the choice of the coordinate system (or the basis) in the space, and thus one should rather say that a set is an Orlicz ball in a given coordinate system, than in and of itself. We shall speak more in the language of functions (that is, instead of talking about the Orlicz ball, we shall consider its characteristic function), which motivates the following definitions:

\begin{Def} By an {\em oriented function} we shall mean a triple $\F = (s, V, \Epsilon)$, where $V$ is a linear space of finite dimension over $\R$, $\Epsilon$ is a basis of $V$ and $f : V_+ \ra \R$. \end{Def} 

\begin{Def} An oriented function $\F = (s,V,\Epsilon)$ is called {\em Orlicz--based} if 
$$s (x_1,\ldots,x_n) = m\Big(\sum_{i=1}^n f_i(x_i) \Big) \prod_{i=1}^n w_i(x_i),$$
where $(f_i)_{i=1}^n$ are Young functions or $f_i \equiv \infty$, $(w_i)_{i=1}^n$ are log--concave functions supported on $\RO$ and $m : \RO \cup \{\infty\} \ra \RO$ is a log--concave function with compact support, attaining its maximum at 0, with $m(\infty) = 0$.
\end{Def}

We shall sometimes speak of $s$ as being a function on the whole of $V$ by extending it by $0$ outside $V_+$.

The characteristic function of the positive generalized quadrant of an generalized Orlicz ball gives the simplest example of an Orlicz--based function.

The proof of the following two simple lemmas is given in \cite{malcin}:

\begin{lemma}\label{rosncalk} Let $\mi$ be any measure on the interval $I$. Let $f, g, h : I \ra \RO$, suppose $\supp f \subset \supp g$ and both $f\slash g$ and $h$ are decreasing on their domains. Then
$$\frac{\int_I f(x) d\mi(x)}{\int_I g(x) d\mi(x)} \leq \frac{\int_I f(x) h(x) d\mi(x)}{\int_I g(x) h(x) d\mi(x)}$$ if both sides are well defined.\end{lemma}

\begin{lemma}\label{rosnmix} Let $\mi$, $I$, $f$ and $g$ satisfy conditions as above. Then for any $a < b \leq d$ and $a \leq c < d$ we have 
$$\frac{\int_a^b f(x)d\mi}{\int_a^b g(x)d\mi} \geq \frac{\int_c^d f(x)d\mi}{\int_c^d g(x)d\mi}$$ if both sides are well defined.
\end{lemma}

\section{A localization type lemma}

The idea given below is similar to the so--called localization lemma proven in the paper \cite{KLS}. This is the part which allows us to circumvent the transfinite induction used in the original proof in \cite{malcin}.

The crucial property of the class of Orlicz--based functions is that it is closed under the following transformations:

\begin{Def} Let $\F = (s,V,\Epsilon)$ be an oriented function $s$. Then we define the {\em sons} of $\F$ as follows:
\begin{itemize}
\item for $i \in \{1,2,\ldots,\dim V\}$ and a log--concave function $w : \R \ra \RO$ the triple $(\tilde{s},V,\Epsilon)$ is a son of $\F$, where $\tilde{s}(x) = s(x) \cdot w(x_i)$,
\item if $H$ is an affine hyperplane in $V$ given by the equation $x_i = ax_j + b$ for some non--negative $a,b$ and $i,j\in \{1,2,\ldots,\dim V\}$, then the triple $(\tilde{s},\tilde{H},\tilde{\Epsilon})$ is a son of $\F$, where $\tilde{H}$ is defined to be $H$ with the linear structure given by setting as the origin the point $x_i = b, x_k = 0$ for $k \neq i$; $\tilde{s}$ is the restriction of $s$ to $\tilde{H}$, and $\tilde{\Epsilon}$ is obtained from $\Epsilon$ by substituting $e_i$ and $e_j$ by $a e_i + e_j$,
\item if $x \in V_+$, then the triple $(\tilde{s},\tilde{V},\Epsilon)$ is a son of $\F$, where $\tilde{V}$ is $V$ with origin fixed at $x$, and $\tilde{s}$ is $s$ restricted to $\tilde{V}$.
\end{itemize}
We define the relation of being a {\em descendant} of an oriented function $\F$ as the smallest transitive and reflexive extension of the relation of being a son.
\end{Def}

\begin{lemma}
Let $\F = (s,V,\Epsilon)$ be an Orlicz--based function. Then any descendant of $\F$ will also be an Orlicz--based function.
\end{lemma}

\begin{proof}
It suffices to consider any son $\F'$ of $\F$. Let the functions $f_i$, $w_i$ and $m$ give a representation of $\F$ as an Orlicz--based function. We shall define the functions $f_i'$, $w_i'$ and $m'$ which represent $\F'$ as an Orlicz--based function. In the first case of the definition (where $\F'$ is created by multiplying $w_i$ by $w$) we take $f_i' = f_i$, $m' = m$ and $w_j' = w_j$ for $j \neq i$, while $w_i' = w \cdot w_i$, and the product of log--concave functions is log--concave. In the second case it suffices to replace $f_i$ and $f_j$ by a single function $f_i(a x + b) + f_j(x) - f_i(b)$ (we assume $\infty - \infty = \infty$), analogously substitute $w_i$ and $w_j$ by a single $w$ and put $m'(x) = m(x + f_i(b))$ to get a representation of $\F'$ as an Orlicz--based function. In the third case, we take $f_i'(t) = f_i(t + x_i) - f_i(x_i)$, $w_i'(t) = w_i(t + x_i)$ and $m'(t) = m(t + \sum f_i(x_i))$.
\end{proof}

We will begin by proving two auxilliary lemmas:

\begin{lemma}\label{krem}
Let $K$ be a convex non--empty compact set in $\R^n$ with a non--empty interior, $K' \subset K$ a convex, compact subset of dimension $k \leq n$ and let $l$ be an affine subspace of dimension $k$ spanned by $K'$. Let $K_m$ be a descending sequence of compact, convex subsets of $K$ with non--empty interiors satisfying $K' = \bigcap_{m=1}^\infty K_m$.
Let $g_m(x)  = \lambda_{n-k}(K_m \cap P^{-1}(x)) \slash \lambda_n(K_m)$ for $x \in l$, where $P$ is the orthogonal projection to $l$. Then there exists a subsequence $K_{m_i}$ of $K_m$ and a log--concave function $g$ with support $K'$ such that $g_{m_i}$ converges almost uniformly to $g$ on $\Int_l K'$ and $\int_{K'} g = 1$. Moreover the family $g_m$ is uniformly bounded on $l$.
\end{lemma}

\begin{proof}
If $k = n$, $g_m$ converge uniformly on $K'$ to $\1_{K'} \slash \lambda_n(K')$. Further on we assume $k < n$.

Let $T_m$ denote the maximum of $g_m$ on $l$, suppose it is attained at the point $O$. Let $\phi_m$ be the Minkowski functional on $l$ given by $\supp g_m$, where $O$ is taken to be the origin (that is $\phi_m(x) = \inf \{\lambda : O + (x-O) \slash \lambda \in \supp g_m\}$). The function $g_m$ is a density of a projection of a uniform measure on an $n$--dimensional convex set to a $k$--dimensional subspace, thus by the Brunn--Minkowski inequality (see eg. \cite{ff}) $\sqrt[n-k]{g_m}$ is concave on its support, thus if $\phi_m(x) \leq 1$, then $g_m(x) \geq T_m(1 - \phi_m(x))^{n-k}$. Further on
\begin{align*} 1 &= \int_{l} g_m(x) d\lambda_k (x) = \int_0^{T_m} \lambda_k \{x : g_m(x) \geq t\} dt \geq \int_0^{T_m} \lambda_k \{x : T_m(1 - \phi_{m}(x))^{n-k} \geq t\} 
\\ & = \int_0^{T_m} \lambda_k \Bigg\{x : \phi_{m}(x) \leq 1 - \Big(\frac{t}{T_m}\Big)^{1\slash (n-k)}\Bigg\} dt = \int_0^{T_m} \lambda_k \Bigg(\bigg(1 - \Big(\frac{t}{T_m}\Big)^{1\slash (n-k)}\bigg)  \supp g_m\Bigg) dt 
\\ & = \int_0^{T_m} \bigg(1 - \Big(\frac{t}{T_m}\Big)^{1\slash (n-k)}\bigg)^k \lambda_k(\supp g_m) dt = {T_m} \lambda_k(\supp g_m) \int_0^1 (1 - s^{1 \slash (n-k)} )^k ds
\\ &  = {T_m}\lambda_k(\supp g_m) c_{n,k} \geq {T_m}\lambda_k(K')c_{n,k}.\end{align*}
Hence $1 \geq T_m \lambda_k(K') c_{n,k}$, thus $T_m \leq 1 \slash (c_{n,k} \lambda_k(K'))$, and so the sequence $g_m$ is uniformly bounded. 

Our aim is to apply the Arzeli--Ascoli Theorem, so we need to prove the almost uniform equicontinuity, meaning uniform equicontinuity on any compact subset of $\Int_l K'$. Let us consider any compact subset $L$ of $\Int_l K'$, from compactness we can choose such a $\delta$, that for any $x \in L$ and $z\notin K'$ we have $|x-z| > \delta$. 

Fix $\Gamma > 1$. For any $x,y \in L$ with $|x-y| < \delta \slash \Gamma$ we can choose such a $z \in K'$, that $\Gamma y = (\Gamma - 1) x + z$. Then
$$g_m(x)^{\frac{1}{n-k}} - \frac{T_m^{\frac{1}{n-k}}}{\Gamma} \leq \frac{\Gamma - 1}{\Gamma} g_m(x)^{\frac{1}{n-k}} \leq \frac{\Gamma - 1}{\Gamma} g_m(x)^{\frac{1}{n-k}} + \frac{1}{\Gamma} g_m(z)^{\frac{1}{n-k}} \leq g_m(y)^{\frac{1}{n-k}},$$
and thus % na mocy (\ref{wujaw})
\begin{align*} g_m(y) - g_m(x) & \leq \frac{T_m^{\frac{1}{n-k}}}{\Gamma}\Big(g_m(y)^{\frac{n-k-1}{n-k}} + g_m(y)^{\frac{n-k-2}{n-k}}g_m(x) + \ldots + g_m(x)^{\frac{n-k-1}{n-k}}\Big) \\
& \leq \frac{(n-k)T_m}{\Gamma} \leq \frac{n-k}{\Gamma c_{n,k} \lambda_k(K')}.\end{align*}
%\frac{1}{g_m(x)^{{n-k}-1} + g_m(x)^{{n-k}-2}g_m(y) + \ldots + g_m(y)^{{n-k}-1}} \\ & \leq \frac{T^{n-k}}{\Gamma + 1} \frac{(2\lambda(K'))^{n-k} (\diam K_0 + \delta)}{(n-k)\delta}.\end{align*}
This expression is independent of $m$, $x$ and $y$, and by choosing an appropriately large $\Gamma$ we can make it arbitrarily small, thus indeed the sequence $(g_m)$ on $L$ is uniformly equicontinuous. Thus by the Arzeli--Ascoli theorem we can choose a subsequence of $g_m$ uniformly convergent on $L$. By choosing a sequence of $L$s increasing to $K'$ we can diagonally construct a subsequence $(g_{m_i})$ almost uniformly convergent on $K'$.  Let $g$ be the limit of $g_{m_i}$ on $\Int_l K'$, extended outside by 0. As $g_m$ were uniformly bounded, by the Lebesgue majorized convergence theorem we have $\int_{K'} g_m \ra \int_{K'} g$, while as $\bigcap K_m = K'$ and $g_m$ are uniformly bounded, $\int_{K'} g_m \ra 1$, thus $\int_{K'} g = 1$. All $g_m$s are log--concave, thus by a simple limit argument $g$ is also log--concave on $\Int_l K'$, which ends the proof.
\end{proof}

\begin{lemma}\label{czastka}
Let $K$ be a compact convex set in $\R^n$ and $K' \subset \R^k$ its convex, compact subset containing a point from the interior of $K$. Let $K_m$ be a decreasing sequence of convex compact subsets of $\R^k \cap K$ of dimension $k$ such that $\bigcap K_m = K'$. Then there exists a log--concave measure $\nu$ on $K'$ and a subsequence $K_{m_i}$ of $K_m$ such that for any continuous function $f$ on $K$ we have
$$\frac{\int_{(K_{m_i} \times \R^{n-k}) \cap K} f(x) d\lambda_n(x)}{\lambda_n((K_{m_i} \times \R^{n-k}) \cap K)} \ra_{m \ra \infty} \frac{\int_{(K' \times \R^{n-k}) \cap K} f(x) d\nu \otimes d\lambda_{n-k}(x)}{\nu \otimes \lambda_{n-k} ((K' \times \R^{n-k}) \cap K)}.$$
\end{lemma}

\begin{proof}
Let $l$ be the affine subspace spanned by $K'$, denote $p = k - \dim l$. 
Let $P$ be the orthogonal projection from $\R^n$ onto $l \times \R^{n-k}$ and denote $g_m(x) = \lambda_p(K_m \cap P^{-1}(x)) \slash \lambda_k(K_m)$ for $x \in l$. Choose a sequence $m_i$ and a function $g$ on $K'$ according to Lemma \ref{krem} so that $g_{m_i}$ converges to $g$ almost uniformly on $\Int_l K'$. For simplicity we pass to the subsequence and assume $g_m$ converges almost uniformly to $g$ on $\Int_l K'$.

Let $A_\delta = (\R^n \setminus K) + \delta B_2^n$, $B_\delta = K + \delta B_2^n$. Notice that if $\1_K(v) \neq \1_K(P(v))$, then $v \in A_{|v - P(v)|} \cap B_{|v - P(v)|}$. Let $O \in K' \cap \Int K$. Fix an $\eps > 0$. Consider the set $K_\eps^- = (1+\eps)^{-1} \big(K - \{O\}\big) + \{O\}$, that is the homothetic image of $K$ with scale $1 \slash (1 + \eps)$ and origin $O$. It is a compact set contained in the interior of $K$, thus for some $\delta_- > 0$ we have $K_\eps^- \cap A_{2\delta_-} = \emptyset$. Similarly if we consider $K_\eps^+ = (1+\eps)\big(K - \{O\}\big) + \{O\}$, then for some $\delta_+ > 0$ we have $K_\eps^+ \supset B_{2\delta_+}$. Let $\delta = \min\{\delta_-,\delta_+\}$, so that if $\1_K(v) \neq \1_K(P(v))$ and $|P(v) - v| < \delta$, then $v \in K_\eps^+ \setminus K_\eps^-$.

Take $m$ so large that for any $x \in K_m$ we have $|x-P(x)| < \delta$. 
Consider $\tilde{K} = (K_m \times \R^{n-k}) \cap K_\eps^-$. For any point of that set $\1_K(v) = \1_K(P(v)) = 1$. On the other hand the set of all points in $K_m \times \R^{n-k}$, for which $\1_K(v) \neq \1_K(P(v))$ is contained in $K_\eps^+ \setminus K_\eps^-$, and thus in particular is contained in $\big((1+\eps)^2 (\tilde{K} - \{O\}) + \{O\}\big) \setminus \tilde{K}$ (here we use the fact that $\tilde{K}$ is convex and contains $O$), so $$\frac{\lambda_n\{x \in K_m\times \R^{n-k} : \1_K(x) \neq \1_K(P(x))\}}{ \lambda_n((K_m \times \R^{n-k}) \cap K)} \leq (1+\eps)^2 - 1,$$ and so the quotient is arbitrarily small for sufficiently large $m$. Now
\begin{align*}
& \frac{\int_{(K_m \times \R^{n-k}) \cap K} f(x) d\lambda_n (x)}{\lambda_n((K_m \times \R^{n-k}) \cap K)} =
\frac{\int_{K_m \times \R^{n-k}} f(x) \1_K(x) d\lambda_n (x)}{\lambda_n((K_m \times \R^{n-k}) \cap K)} \\
& =
\frac{\int_{K_m \times \R^{n-k}} f(x) \1_K(P(x)) d\lambda_n (x)}{\lambda_n((K_m \times \R^{n-k}) \cap K)} + \frac{\int_{K_m \times \R^{n-k}} f(x) (\1_K(x) - \1_K(P(x))) d\lambda_n (x)}{\lambda_n((K_m \times \R^{n-k}) \cap K)}.
\end{align*}
The second summand can be bounded by $\sup f$ times $\lambda_n\{x \in K_m\times \R^{n-k} : \1_K(x) \neq \1_K(P(x))\} \slash \lambda_n((K_m \times \R^{n-k}) \cap K)$, and so it converges to zero. We have to bound the first summand. Let
\begin{align*}
I_m & := \frac{\int_{K_m \times \R^{n-k}} f(v) \1_K(P(v)) d\lambda_n (v)}{\lambda_n((K_m \times \R^{n-k}) \cap K)} \\  
= &\frac{\int_{\R^{n-k}} \int_{l} \1_K(x,y,0) \lambda_{p} (P^{-1}(y) \cap K_m) \big[ \int_{P^{-1}(y) \cap K_m} \frac{f(x,y,z)}{\lambda_{p} (P^{-1}(y) \cap K_m)} d\lambda_{p} (z)\big] d\lambda_{k-p} (y) d\lambda_{n-k} (x)}{\lambda_n((K_m \times \R^{n-k}) \cap K)},
\end{align*}
where by $P^{-1}(y)$ we mean the counterimage of $y \in l \subset \R^n$ with respect to the projection $P$.
Let $$f_m(x,y) = \int_{P^{-1}(y) \cap K_m} \frac{f(x,y,z)}{\lambda_{p} (P^{-1}(y) \cap K_m)} d\lambda_{p} (z).$$ The function $f_m$ is an average of $f$ on the set $P^{-1}(y) \cap K_m$. Recall $f$ is continuous on $K$, so it is uniformly continuous, and the diameter of the set $P^{-1}(y) \cap K_m$ converges uniformly (with respect to $y$) to zero with $m \ra \infty$, thus $f_m(x,y)$ converges uniformly to $f(x,y,0)$. By inserting the definition of $g_m$ we get
$$
I_m = \frac{\lambda_k(K_m)}{\lambda_n((K_m \times \R^{n-k}) \cap K)}
\int_{\R^{n-k}} \int_{l}  \1_K(x,y,0) g_m(y) f_m(x,y) d\lambda_{k-p} (y) d\lambda_{n-k} (x).
$$
The functions $f_m$  and $g_m$ are uniformly bounded and $\lambda_{k-p}(\supp g_m \setminus K') \ra 0$, thus $$\int_{\R^{n-k}} \int_{l \setminus K'} \1_K(x,y,0) g_m(y) f_m(x,y) d\lambda_{k-p}(y) d\lambda_{n-k}(x)$$ converges to zero. We have to estimate
$$
\frac{\lambda_k(K_m)}{\lambda_n((K_m \times \R^{n-k}) \cap K)}
\int_{(K' \times \R^{n-k}) \cap K} g_m(y) f_m(x,y) d\lambda_{k-p} (y) d\lambda_{n-k} (x).
$$
Both $g_m$ and $f_m$ are uniformly bounded and almost uniformly convergent, thus $g_m(y) f_m(x,y)$ converges almost uniformly to  $g(y) f(x,y,0)$ on $K' \times \R^{n-k}$. Let $\nu$ denote the measure on $l$ with density $g \1_{K'}$. Then
\begin{align*}\int_{(K' \times \R^{n-k}) \cap K} g_m(y) f_m(x,y) d\lambda_{k-p}(y) d\lambda_{n-k}(x) & \ra_{m\ra \infty} \int_{(K' \times \R^{n-k}) \cap K} g(y) f(x,y,0) d\lambda_{k-p}(y) d\lambda_{n-k}(x) \\ = & \int_{(K' \times \R^{n-k}) \cap K} f(x,y) d\lambda_{n-k}(x) \otimes d\nu(y).\end{align*}

Moreover 
\begin{align}
\nonumber & \frac{\lambda_n((K_m \times \R^{n-k}) \cap K)}{\lambda_k(K_m)}  = 
\frac{\int_{K_m \times \R^{n-k}} \1_K(v) d\lambda_n(v)}{\lambda_k(K_m)} 
 \\ = & \frac{\int_{K_m \times \R^{n-k}} \1_K(P(v))d\lambda_n(v)}{\lambda_k(K_m)} +  \frac{\int_{K_m \times \R^{n-k}} \1_K(v) - \1_K(P(v))d\lambda_n(v)}{\lambda_n((K_m \times \R^{n-k}) \cap K)}\frac{\lambda_n((K_m \times \R^{n-k}) \cap K)}{\lambda_k(K_m)}\label{ddfa} \end{align}
Notice that 
$$
\frac{\lambda_n((K_m \times \R^{n-k}) \cap K)}{\lambda_k(K_m)} \leq \frac{\lambda_n(K_m \times P_{n-k}(K))}{\lambda_k(K_m)} = \lambda_{n-k}(P_{n-k}(K)),
$$
where $P_{n-k}$ is the orthogonal projection onto $\R^{n-k}$. Thus in the second summand of (\ref{ddfa}) the second fraction is bounded, while the first converges to zero, which we already proved. As to the first summand, 
\begin{align*}
&\frac{\int_{K_m \times \R^{n-k}} \1_K(P(v))d\lambda_n(v)}{\lambda_k(K_m)} = \int_{l \times \R^{n-k}} g_m(x) \1_K(x,y,0) d\lambda_{k-p}(x) d\lambda_{n-k}(y) \\
 \ra_{m \ra \infty}& \int_{l \times \R^{n-k}} g(x) \1_K(x,y,0) d\lambda_{k-p}(x) d\lambda_{n-k}(y) = 
\nu \otimes \lambda_{n-k} ((K' \times \R^{n-k}) \cap K),
\end{align*}
where the convergence follows from the almost uniform convergence of the integrand, which ends the proof of the Lemma. 
\end{proof}

The following definitions will be useful:
\begin{Def} A set $K \subset \R^2$ is called {\em spanned} by the points $a,b$ if $K$ is convex, compact, $a,b\in K$, and for any $x \in K$ we have $a \sleq x \sleq b$. A set is called spanned if it is spanned by some two points $a,b$.\end{Def}
Geometrically this definition means that $K$ is convex, compact and if we inscribe $K$ in a rectangle with edges parallel to the coordinate axes, then the lower left corner and the upper right corner of the rectangle are contained in $K$.

\begin{Def}
For a linear space $V$ with a basis $\Epsilon$ by a {\em splitting} of $V$ with respect to $\Epsilon$ we mean such a decomposition $V = V_1 \oplus V_2$ and $\Epsilon = \Epsilon_1 \cup \Epsilon_2$ that $\Epsilon_i$ is a basis of $V_i$.
\end{Def}

\begin{Def}
Consider an Orlicz--based function $\F = (s,V,\Epsilon)$ and functions $f,g : V \ra \R$. We shall say that $\F,f$ and $g$ {\em satisfy the $\Theta$ condition}, if for any splitting $V = V_1 \oplus V_2$ with respect to $\Epsilon$ and any $0 \sleq x \sleq y \in V_1$ we have
\begin{equation} \label{zalozenieKLS} 
\frac{\int_{V_2} f(x,z) s(z) d\lambda_{n-k}(z)}{\int_{V_2} g(x,z) s(z) d\lambda_{n-k}(z)} \geq \frac{\int_{\R^{n-k}} f(y,z) s(z) d\lambda_{n-k}(z)}{\int_{\R^{n-k}} g(y,z) s(z) d\lambda_{n-k}(z)},\end{equation} 
whenever both sides are well--defined, where $n$ denotes $\dim V$ and $k$ --- $\dim V_1$.

We shall say that an Orlicz--based function $\F$ and functions $f,g$ satisfy the {\em hereditary $\Theta$ condition} if any descendant $\F' = (s', V', \Epsilon')$ of $\F$ and the restrictions of $f$ and $g$ to the $V'$ satisfy the $\Theta$ condition.
\end{Def}

\begin{lemma}\label{transfin}
Consider an Orlicz--based function $\F = (s,\R^n,\Epsilon)$, where $\Epsilon$ is the standard basis in $\R^n$, and three continuous functions --- $f,g : \R^n \ra [0,M]$ and $h : \R^n \ra [\eps,M]$ for some $M > \eps > 0$, where $\{f > 0\}\subset \{g > 0\}$, $\int_{\R^n} s(x) dx = 1$ and $\supp s \subset \supp g$. Assume that $\M,f$ and $g$ satisfy the hereditary $\Theta$ condition.
Additionally assume that
\begin{equation} \label{ogolnaKLS} \frac{\int_{\RO^n} f(z) h(z) s(z) dz}{\int_{\RO^n} g(z) h(z) s(z) dz} < \frac{\int_{\RO^n} f(z) s(z) dz}{\int_{\RO^n} g(z) s(z) dz}.
\end{equation}
Then there exist two different points $a \sleq b \in \R^n$ and a log--concave measure $\nu$ on the interval $I = [a,b]$, such that
$$\frac{\int_{\RO^n} f(x) s(x) dx}{\int_{\RO^n} g(x) s(x) dx} = \frac{\int_I f(x) s(x) d\nu(x)}{\int_I g(x) s(x) d\nu(x)} > \frac{\int_I f(x) h(x) s(x) d\nu(x)}{\int_I g(x) h(x) s(x) d\nu(x)}.$$ In particular $h$ cannot be coordinate--wise non--increasing.
\end{lemma}

\begin{proof}
We shall proceed by induction upon dimension. For $n = 0$ the condition (\ref{ogolnaKLS}) cannot be satisfied. For $n = 1$ no assumptions are needed, the interval $\supp s$ with the Lebesgue measure satisfies the conditions of the Lemma. Let us consider higher $n$. We shall a construct a decreasing sequence of spanned sets  $K_0 \supset K_1 \supset \ldots$ in $\spa\{e_1,e_2\}$ satisfying the following four conditions:
\begin{equation} \label{Kwar4} \int_{K_m \times\RO^{n-2}} s(x) dx > 0, \end{equation}
\begin{equation}\label{Kwar1} \frac{\int_{K_i \times \RO^{n-2}} f(x) s(x) dx}{\int_{K_i\times \RO^{n-2}} g(x) d(x) dx} = \frac{\int_{\RO^n} f(x) s(x) dx}{\int_{\RO^n} g(x) s(x) dx}\end{equation}
\begin{equation}\label{Kwar2} \frac{\int_{K_i\times \RO^{n-2}} \tf(x) h(x) s(x) dx}{\int_{K_i\times \RO^{n-2}} g(x) h(x) s(x) dx} \leq \frac{\int_{\RO^n} \tf(x) h(x) s(x) dx}{\int_{\RO^n} g(x) h(x) s(x) dx} < \frac{\int_{\RO^n} f(x) s(x) dx}{\int_{\RO^n} g(x) s(x) dx},\end{equation}
\begin{equation} \label{Kwar3} \bigcap_{m=0}^\infty K_m \mbox{ is an interval or a point.}\end{equation}
The function $\tf$ is a slight modification of $f$, which ensures our sequence does not approach the edge of $\supp s$ too closely. Choose $M'$ so that
\begin{equation}\label{Mprim}
\frac{M'}{M} > \frac{\int_{\RO^n} f(x) s(x) dx}{\int_{\RO^n} g(x) s(x) dx}
\end{equation}
and $\tilde{c} > 0$ so that
\begin{equation}\label{tildec}
\frac{\tilde{c} M'M + \int_{\RO^n} f(x) h(x) s(x) dx}{\int_{\RO^n} g(x) h(x) s(x) dx} < \frac{\int_{\RO^n} f(x) s(x) dx}{\int_{\RO^n} g(x) s(x) dx}.
\end{equation}
Let $A_t := (\RO^n \setminus \supp s) + t \Int B_2^n$, that is the set of points which are less than $t$ away from the edge of $\supp s$. Fix $\eps > 0$ so that $\int_{A_\eps}s(x) dx < \tilde{c}$. Let $\Delta f$ be a continuous function which is equal to $M'$ on $A_{\eps \slash 2}$, equal to 0 on $\RO^n \setminus A_\eps$ and is bounded from below by zero, and from above by $M'$ (such a function exists for instance by the Urysohn Lemma). Then $\tilde{f} := f + \Delta f$. Notice that the second inequality of condition (\ref{Kwar2}) is satisfied by (\ref{tildec}), and that if the set $L$ is contained in $A_{\eps \slash 2}$ then by (\ref{Mprim}) and (\ref{tildec}) the following inequality is satisfied
$$
\frac{\int_{L} \tf(x) h(x) s(x) dx}{\int_{L} g(x) h(x) s(x) dx} \geq
\frac{M'}{M} > \frac{\int_{\RO^n} f(x) s(x) dx}{\int_{\RO^n} g(x) s(x) dx} > 
\frac{\int_{\RO^n} \tf(x) h(x) s(x) dx}{\int_{\RO^n} g(x) h(x) s(x) dx}.$$

For $K_0$ we can take any rectangle in $\spa\{e_1,e_2\}$ with edges parallel to the coordinate axes and containing the projection of $\supp s$ onto $\spa \{e_1,e_2\}$. We order all the points with both coordinates rational into a sequence $(q_i)_{i=1}^\infty$. Having $K_m$ we will want to construct $K_{m+1}$. Let $O_m$ be the first point from the sequence $(q_i)$ contained in the interior of $K_m$ (by (\ref{Kwar4}) $K_m$ is a convex set of positive measure, and thus contains a point with both coordinates rational). Consider a vertical (ie. parallel to $e_2$) passing through $O_m$, by $K_E$ denote the part of $K_m$ to the right of that line, by $K_W$ the part to the left. Further on we shall prove the Lemma \ref{horline}, which will show that under the assumptions of our lemma we have
$$\frac{\int_{K_W\times \RO^{n-2}} f(x) s(x) dx}{\int_{K_W\times \RO^{n-2}} g(x) s(x) dx} \geq \frac{\int_{K_E\times \RO^{n-2}} f(x) s(x) dx}{\int_{K_E\times \RO^{n-2}} g(x) s(x) dx},$$
if both sides are well defined. If one of the sides is not well defined (say the one corresponding to $K_E$), then $(K_E\times \RO^{n-2}) \cap \supp s$ has measure zero, so we can set $K_{m+1} = K_W$ --- all integrals on $K_W\times \RO^{n-2}$ will be equal to the corresponding integrals on $K_m\times \RO^{n-2}$, so as $K_m$ satisfied (\ref{Kwar4}), (\ref{Kwar1}) and (\ref{Kwar2}), $K_W$ also satisfies them. We shall check the condition (\ref{Kwar3}) further on. Thus assume both sides are well--defined. From this we know that
\begin{equation}\frac{\int_{K_E\times \RO^{n-2}} f(x) s(x) dx}{\int_{K_E\times \RO^{n-2}} g(x) s(x) dx} \leq \frac{\int_{K_m\times \RO^{n-2}} f(x) s(x) dx}{\int_{K_m\times \RO^{n-2}} g(x) s(x) dx} = \frac{\int_{\RO^n} f(x) s(x) dx}{\int_{\RO^n} g(x) s(x) dx}.\label{east}\end{equation}
Similarly, when we consider a horizontal line through $O_m$, dividing $K_m$ into the upper part $K_N$ and lower part $K_S$, Lemma \ref{horline} will give
\begin{equation}\frac{\int_{K_S\times \RO^{n-2}} f(x) s(x) dx}{\int_{K_S\times \RO^{n-2}} g(x) s(x) dx} \geq \frac{\int_{\RO^n} f(x) s(x) dx}{\int_{\RO^n} g(x) s(x) dx},\label{south}\end{equation}
again we can assume the left side is well--defined.

If we will rotate clockwise a line passing through $O_m$ in a continuous fashion from the vertical position to the horizontal, and divide $K_m$ into two parts $K_+$ and $K_-$, then the integrals $\int_{K_+\times \RO^{n-2}} f(x) s(x) dx$ and $\int_{K_+\times \RO^{n-2}} g(x) s(x) dx$ will change continuously. If for any of the intermediate positions of the line the second of these integrals will be equal to zero, we can take $K_{m+1} = K_-$ as previously. If not, then their quotient changes continuously. For the vertical line $K_+ = K_E$, so by (\ref{east}) the quotient is no larger than for the whole $K_m$. For the horizontal line $K_+ = K_S$, thus by (\ref{south}) the quotient is no smaller than for the whole $K_m$. Thus by the Darboux property there exists a division of $K_m$ into two sets $K_+$ and $K_-$, both of which satisfy  (\ref{Kwar1}) and (\ref{Kwar4}). Both those sets are spanned.

Notice that at least one of these sets has to satisfy condition (\ref{Kwar2}) --- if both of them did not, then $K_m$ could not satisfy it either. Let $K_{m+1}$ be such a set.

Obviously $K_{m+1} \subset K_m$ and $K_{m+1}$ is a spanned set.

Now let us consider condition (\ref{Kwar3}). Notice that if a point $q$ is used as the point $O_m$ for some $m$, then it will lie on the edge of $K_{m+1}$, and thus will not lie in the interior of any $K_l$ for $l > m$, and so will not be re--used as $O_l$ for $l > m$. Thus no point with both coordinates rational lies in the interior of $K_\infty := \bigcap_{m=0}^\infty K_m$. Moreover $K_\infty$ is an intersection of a family of convex sets, and thus a convex set, so it has to be an interval or a point (if it contained three affinely independent points, it would contain their convex hull, and inside it a point with both coordinates rational). 

Consider the set $(K_\infty \times \R^{n-2}) \cap \supp s$. Notice that all $K_m$ satisfied in particular condition (\ref{Kwar2}), and thus by the definition of $\tilde{f}$ none of them is contained in  $A_{\eps \slash 2}$, and thus $K_\infty$ cannot be contained in $A_{\eps \slash 2}$ --- thus it contains a point from the interior of $\supp s$. Let $H$ be the minimal affine subspace containing $K_\infty \times \R^{n-2}$. Let $l$ be the affine subspace spanned by $K_\infty$ (and thus a line or a point). We shall apply Lemma \ref{czastka} taking $K' = K_\infty$. We will obtain some subsequence $m_i$ and a log--concave measure $\nu$ on $l$ supported on $K_\infty$. The functions $f,g,h$ and $s$ are continuous on $\supp s$, and thus all the integrals on $K_{m_i} \times \R^{n-2}$ in the condition (\ref{Kwar2}) converge to appropriate integral on $K_\infty \times \R^{n-2}$, and thus by the condition (\ref{Kwar2}), condition (\ref{ogolnaKLS}) holds for $H$. The function $s$ restricted to $H$ and multipied by the density of $\nu$ is a descendant of $s$ ($H$ can be given by either $e_1 = ae_2 + b$ or $e_2 = ae_1 + b$, as $K_\infty$ is a spanned set), and thus an Orlicz--based function, and the hereditary $\Theta$ condition for the restrictions of $f$ and $g$ to $H$ is trivially satisfied. Thus by the induction hypothesis there exists an interval $I$ in $H$ and a measure $\nu$ as in the thesis of the lemma --- and this interval and measure satisfy the thesis of the lemma also for $\R^n$, which ends the proof in the general case.

If $h$ was coordinate--wise decreasing, restricting $f$, $g$ and $h$ to $I$ we would obtain a contradiction with Lemma \ref{rosncalk} -- $f\slash g$ is decreasing on $I$ by the $\Theta$ condition and $h$ is decreasing on $I$, so inequality (\ref{ogolnaKLS}) cannot hold.
\end{proof}

To end the proof we only need Lemma \ref{horline}, which describes the behaviour of the proportion of integrals of $f$ and $g$ with the assumption of the hereditary $\Theta$ condition when we divide a spanned set by a horizontal or vertical line:

\begin{lemma} \label{horline}
Let $\F = (s,\R^n,\Epsilon), f$ and $g$ be as in the assumptions of Lemma \ref{transfin}. Let $K$ be such a spanned set in $\spa\{e_1,e_2\}$ that $\int_{K\times \R^{n-2}} s(x) dx > 0$. Let $K_x = K \cap \{v : \is{e_1}{v} = x\}$ be the intersection of the set $K$ with a vertical line. Let 
$$\Theta(x) = \frac{\int_{K_x \times \R^{n-2}} f(v) s(v) d\lambda_{n-1}(v)}{\int_{K_x \times \R^{n-2}} g(v) s(v)d\lambda_{n-1}(v)}.$$ Then $\Theta(x)$ is decreasing on its domain. In particular, if the line $\is{e_1}{v} = x_0$ divides $K$ into two parts, $K_- = \{v \in K : \is{e_1}{v} \leq x_0\}$ and $K_+ = \{v \in K : \is{e_1}{v} \geq x_0\}$, then $$\frac{\int_{K_- \times \R^{n-2}} f(v) s(v) dv}{\int_{K_- \times \R^{n-2}} g(v) s(v) dv} \geq \frac{\int_{K_+ \times \R^{n-2}} f(v) s(v) dv}{\int_{K_+ \times \R^{n-2}} g(v) s(v) dv},$$ if both sides are well defined.\end{lemma}

\begin{proof}
From the $\Theta$ property for any $y_0$ the function  
$$x \mapsto \frac{\int_{\{x\}\times \{y_0\} \times \R^{n-2}} f(v) s(v)d\lambda_{n-2}(v)}{\int_{\{x\} \times \{y_0\} \times \R^{n-2}} g(v) s(v)d\lambda_{n-2}(v)}$$ is decreasing where well--defined. The support of $s$ is convex, $K$ is also convex, $\supp g \supset \supp s$, and thus the domain of this function is an interval. Thus by Lemma \ref{rosnmix}, we obtain
$$\frac{\int_{\{x\} \times [y_a,y_b] \times \R^{n-2}} f(v) s(v)d\lambda_{n-1}(v)}{\int_{\{x\} \times [y_a,y_b] \times \R^{n-2}} g(v) s(v)d\lambda_{n-1}(v)} \geq \frac{\int_{\{x\} \times [y_c,y_d] \times \R^{n-2}} f(v) s(v)d\lambda_{n-1}(v)}{\int_{\{x\} \times [y_c,y_d] \times \R^{n-2}} g(v) s(v)d\lambda_{n-1}(v)},$$ as long as $y_a < y_b < y_d$ and $y_a < y_c < y_d$ and both sides are well--defined.

The second property we need is 
$$\frac{\int_{\{x_1\} \times [y_a,y_b] \times \R^{n-2}} f(v) s(v)d\lambda_{n-1}(v)}{\int_{\{x_1\} \times [y_a,y_b] \times \R^{n-2}} g(v) s(v)d\lambda_{n-1}(v)} \geq \frac{\int_{\{x_2\} \times [y_a,y_b] \times \R^{n-2}} f(v) s(v)d\lambda_{n-1}(v)}{\int_{\{x_2\} \times [y_a,y_b] \times \R^{n-2}} g(v) s(v)d\lambda_{n-1}(v)},$$ as long as $x_2 > x_1$ and both sides are well--defined. To obtain this, notice that by moving the origin to $(0,y_a,0,\ldots,0)$ and multiplying $s$ by $\1_{y \leq y_b}$ we obtain a descendant of $\F$, and thus the hereditary $\Theta$ condition guarantees in particular
$$\frac{\int_{\{x_1\}\times\R \times \R^{n-2}} f(v) m(v) s(v)d\lambda_{n-1}(v)}{\int_{\{x_1\} \times\R \times \R^{n-2}} g(v) m(v) s(v)d\lambda_{n-1}(v)} \geq \frac{\int_{\{x_2\}\times\R \times \R^{n-2}} f(v) m(v) s(v)d\lambda_{n-1}(v)}{\int_{\{x_2\}\times\R \times \R^{n-2}} g(v) m(v) s(v)d\lambda_{n-1}(v)},$$ which gives the thesis.

Now notice that as $K$ is spanned, then for $x_2 > x_1$ we have $K_{x_1} = \{x_1\} \times [y_a,y_b]$, $K_{x_2} = \{x_2\} \times [y_c,y_d]$ and $y_a < y_b < y_d$ and $y_a < y_c < y_d$, which gives the first part of the thesis. The second follows from the first and Lemma \ref{rosnmix}.
\end{proof}

\section{Negative association of absolute values for Orlicz balls}

We shall use Lemma \ref{transfin} to prove negative association of absolute values for Orlicz balls. This section is based on \cite{malcin}. We shall need a pair of functions satisfying the $\Theta$ condition.

\begin{lemma}\label{dobratheta}
Let $\F = (s,V,\Epsilon)$ be an Orlicz--based function, and let $V = W \times \R$, $\Epsilon = \Epsilon' \cup \{e_n\}$ be a splitting of $V$ with respect to $\Epsilon$. Let $f(x) = s(x,z_2)$ and $g(x) = s(x,z_1)$ for some numbers $0 < z_1 < z_2$ and $x \in W$, where $\supp g$ is non--empty. Let $\G = (t,W,\Epsilon')$ be an oriented function, where $$t(x) := \1_{\supp g}(x) \prod_{i=1}^{\dim W} u_i(x_i)$$ for some log--concave functions $u_i$. Then $\G$ is an Orlicz--based function and $\G$, $f$ and $g$ satisfy the hereditary $\Theta$ condition.
\end{lemma}

\begin{proof}
Let $n = \dim V$. First we shall check that $\G$ is an Orlicz--based function. Let $(w_i)_{i=1}^n$, $(f_i)_{i=1}^n$ and $m$ be functions certifying that $\F$ is an Orlicz--based function. Consider the following functions on $W$:  $(u_i \phi(w_i))_{i=1}^{n-1}$, $(f_i)_{i=1}^{n-1}$ and $x \mapsto \phi(m(x + f_n(z_1)))$, where $\phi = \1_{(0,\infty)}$. These functions give $\G$ as an Orlicz--based function.

We proceed to prove the $\Theta$ condition. We will consider the splitting $W = W_1 \oplus W_2$, let $\dim W_1 = k$. We want to check the function $$\frac{\int_{W_2} f(x,y) t(x,y) d\lambda_{n-k-1}(y)}{\int_{W_2} g(x,y) t(x,y) d\lambda_{n-k-1}(y)}$$ is coordinate--wise non--increasing on $W_1$. Obviously it suffices to change one coordinate at a time, keeping the others fixed. Notice that fixing the coordinate $x_i$ is equivalent to intersecting $W$ with a subspace given by $x_i = b$ and substituting $\F$ and $\G$ by their appropriate descendants (and, correspondingly, $f$ and $g$ by their appropriate restrictions), and thus by simple induction upon $\dim W_1$ it suffices to consider the case $\dim W_1 = 1$. Without loss of generality we can identify $W_2$ with $\R^{n-2}$. Thus it suffices to prove
\begin{align}\label{dupa12}
&\int_{\R^{n-2}} s(x_1,y,z_1) t(x_1,y)  d\lambda_{n-2}(y) \int_{\R^{n-2}} s(x_2,y,z_2) t(x_2,y) d\lambda_{n-2}(y) \leq \\ \nonumber & 
\int_{\R^{n-2}} s(x_1,y,z_2) t(x_1,y) d\lambda_{n-2}(y) \int_{\R^{n-2}} s(x_2,y,z_1) t(x_2,y) d\lambda_{n-2}(y).
\end{align} Notice that as $$s(x_i,y,z_j) = m(f_1(x_i) + f_2(y_1) + \ldots + f_{n-1}(y_{n-2}) + f_n(z_j)) w_1(x_i) w_n(z_j) \prod_{i=2}^{n-1} w_i(y_{i-1}),$$ in the inequality (\ref{dupa12}) the expressions $w_1(x_i)$ and $w_n(z_j)$ cancel out. Similarly the $u_1(x_i)$ expressions in $t$ cancel out, we can also drop the $\1_{\supp g}$ factor from $t$ as all the integrands disappear outside $\supp g$. Now consider $$r(y_0,y_1,\ldots,y_{n-2}) = m\Big(y_0 + f_2(y_1) + f_3(y_2) + \ldots + f_{n-1}(y_{n-2})\Big)\prod_{i=1}^{n-2} w_{i+1}(y_i) u_{i+1}(y_i).$$ This function is log--concave, as the composition of an increasing convex function with a convex function is convex and the product of log--concave functions is log--concave. Let $P_x = \int_{\R^{n-2}} r(x,v) d\lambda_{n-2}(v)$. As $r$ is log--concave, by the Prekopa--Leindler inequality (see \cite{ff}) we have $P_x^t P_y^{1-t} \leq P_{tx + (1-t)y}$. Let $a = f_1(x_1) + f_n(z_1)$, $b = f_1(x_2) - f_1(x_1)$ and $c = f_n(z_2) - f_n(z_1)$. In particular 
$$P_{a}^{c\slash(b+c)} P_{a+b+c}^{b\slash(b+c)} \leq P_{a+b},$$
$$P_{a}^{b\slash(b+c)} P_{a+b+c}^{c\slash(b+c)} \leq P_{a+c},$$ 
and thus
$$P_a P_{a+b+c} \leq P_{a+b} P_{a+c},$$
which proves inequality (\ref{dupa12}), and thus the thesis.

Finally, we want to prove the hereditary $\Theta$ condition. Let $\G' = (t', W', \Epsilon')$ be a descendant of $\G$, we will want to construct an Orlicz--based function $\F' = (s', W' \times \R, \Epsilon' \cup \{e_0\})$ and log--concave functions $u_i'$ so that $f$ restricted to $W'$ is equal to $s
(\cdot, z_1)$, $g$ restricted to $W'$ is $s(\cdot, z_2)$ and $t' = \supp g \prod u_i'$. 

The construction will proceed by induction upon the descendant hierarchy. Thus suppose $\G'$ is a son of $\G$. We have to consider three cases:
\begin{itemize}
\item In the first case, where $t'(x) = t(x) w(x_i)$ take $u_i' = u_i \cdot w$, and leave all other parameters unchanged.
\item In the second case, where $W'$ is given by $ax_j + b$, take as $\F'$ the son of $\F$ given by the same equation $x_i = ax_j + b$, and replace $u_i$ and $u_j$ by a single $u(t) = u_i(at + b) u_j(t)$.
\item In the third case, if the origin was moved to $x$, consider the son of $\F$ given by moving the origin to $(x,0)$, and functions $u_i'(t) = u_i(t + x_i)$.
\end{itemize}
In each case we can proceed with the inductive construction, and having constructed $\F'$ and $u_i'$ we apply the previous result.
\end{proof}

\begin{lemma}\label{slujsto} Let $\F = (s, V, \Epsilon)$ be an Orlicz--based function, and let $V = W \times \R$, $\Epsilon = \Epsilon' \cup \{e_n\}$ be a splitting of $V$ with respect to $\Epsilon$. Let $h:W_+ \ra \R$ be a bounded coordinate--wise decreasing function, and let $0 < z_1 < z_2$. Let $\G = (t, W, \Epsilon)$ be an oriented function, where $$t(x) = \1_{\{s(x,z_1) > 0\}} \prod u_i(x_i),$$ where $u_i$ are log--concave functions. Then  
$$\frac{\int_{W_+} h(x) s(x,z_2) t(x) d\lambda_{n-1}(x)}{\int_{W_+} h(x) s(x,z_1) t(x) d\lambda_{n-1}(x)} \geq \frac{\int_{W_+} s(x,z_2) t(x) d\lambda_{n-1}(x)}{\int_{W_+} s(x,z_1) t(x) d\lambda_{n-1}(x)}$$ if both sides are well defined.\end{lemma}

\begin{proof}
Suppose the thesis does not hold. Then for some fixed $h$ 
\begin{equation}\label{slabaostra}\frac{\int_{W_+} h(x) s(x,z_2) t(x) d\lambda_{n-1}(x)}{\int_{W_+} h(x) s(x,z_1) t(x) d\lambda_{n-1}(x)} < \frac{\int_{W_+} s(x,z_2) t(x) d\lambda_{n-1}(x)}{\int_{W_+} s(x,z_1) t(x) d\lambda2_{n-1}(x)}.\end{equation} We would like to apply Lemma \ref{transfin}. The role of $f$ will be taken by $s(x,z_2)$, the role of $g$ --- by $s(x,z_1)$. For this we need $s(x,z_1), s(x,z_2)$ and $h$ to be continuous and $h$ to be bounded uniformly away from zero. First notice, that if inequality \ref{slabaostra} holds, then it will also hold if we substitute $h(x) + C$ for $h$. Thus we may assume $h$ is strictly larger than, say, one. Since we have a sharp inequality in (\ref{slabaostra}), it will also hold after a small enough modification of $s$ and $h$. Let $s(x) = m(\sum f_i(x_i)) \prod w_i(x_i)$. First we approximate $m$ from above by a decreasing sequence $m_k$ of continuous log--concave functions with maxima at zero, which converges pointwise to $m$. Then $s_k(x,z_i)$ converges monotonously to $s(x,z_i)$, and thus all the integrals in the inequality (\ref{slabaostra}) converge and we can choose such a $k$, that after substituting $s_k$ for $s$ the  inequality (\ref{slabaostra}) still holds. Similarly we can approximate $f_i$ from below by continuous Young functions (ie. functions that do not jump to $\infty$). Then $s_k$ will still be an Orlicz--based function, and $f$ and $g$ will be continuous and satisfy $\supp f \subset \supp g$ (as $f \leq g$, as both the Young functions and $m$ are increasing). Similarly we approximate $h$ from above by a sequence of continuous functions $h_k$, decreasing coordinate--wise and uniformly bounded away from zero and pointwise convergent to $h$, and substitute $h$ by a sufficiently close approximation $h_k$. We can assume that after these modifications the inequality (\ref{slabaostra}) still holds.

After those modifications the assumptions of Lemma \ref{transfin} are satisfied (the hereditary $\Theta$ condition holds by Lemma \ref{dobratheta}, as $s_k$ is an Orlicz--based function, and condition (\ref{ogolnaKLS}) is simply the inequality (\ref{slabaostra}). Thus the thesis of Lemma holds --- but we assumed $h_k$ to be coordinate--wise decreasing, which contradiction end the proof of our lemma.
\end{proof}

Notice that from the above lemma we obtain by switching sides that the function
$$\frac{\int_{\RO^k} h(x) s(x,z) dx}{\int_{\RO^k} s(x,z) dx}$$ is coordinate--wise decreasing as a function of $z$ for any coordinate--wise decreasing function $h$ and any Orlicz--based function $\F$ on $\R^n$. Thus we can prove the following corollary just as we proved the last part of Lemma \ref{dobratheta}:

\begin{cor}\label{ftc} Consider $\F$, $\G$, and $h$ defined as above. Then the Orlicz--based function $\G$ and functions  
$\int_{W_+} h(x) s(x,z_1) t(x) d\lambda_{n-1}(x)$ and $\int_{W_+} s(x,z_1) t(x) d\lambda_{n-1}(x)$
satisfy the hereditary $\Theta$ condition.
\end{cor}

Now we can prove our thesis in full generality:

\begin{lemma}  \label{ujsto} Let $\f = (s,\R^n,\Epsilon)$ be an Orlicz--based function, let $h: \RO^k \ra \R$ and $\bh: \RO^{n-k} \ra \R$ be coordinate--wise decreasing functions. Then
\begin{equation}\label{costamcos}\frac{\int_{\RO^{n-k}} \bh(x) \int_{\RO^k} h(y) s(x,y) dy dx}{\int_{\RO^{n-k}} \bh(x) \int_{\RO^k} s(x,y) dy dx} 
\leq
\frac{\int_{\RO^{n-k}} \int_{\RO^k} h(y) s(x,y) dy dx}{\int_{\RO^{n-k}} \int_{\RO^k} s(x,y) dy dx}.
\end{equation}\end{lemma}

The proof will be almost identical to the proof of Lemma \ref{slujsto}:
\begin{proof} 
Again we shall aplly Lemma \ref{transfin}. Assume an opposite inequality holds. The role of the function $h$ will be taken, as in Lemma \ref{slujsto}, by a continuous, uniformly bounded from below and coordinate--wise decreasing approximation of $\bh$. We define $f(x) = \int_{\RO^k} h(y) s(x,y) dy$ and $g(x) = \int_{\RO^k} s(x,y) dy$. Approximating $m$, $f_i$ and $\bh$ by continuous functions as in \ref{slujsto} we shall obtain continuous modifications of $f$ and $g$, for which (\ref{ogolnaKLS}) still holds. The hereditary $\Theta$ condition is satisfied by Corollary \ref{ftc}. However our function $h$ is coordinate--wise decreasing, which contradicts Lemma \ref{transfin}. The contradictions shows inequality (\ref{costamcos}) must hold, which ends the proof.\end{proof}

From the above we immediately obtain the main theorem of this paper:
\begin{thm} \label{orliczglownetw}\label{orliczglownetwierdzenie}
Let $f_i$ be Young functions and let $m: \RO \ra \RO$ be any log--concave non--increasing function. Assume that the measure on $\R^n$ with density $m(\sum f_i(|x_i|)$ is probabilistic, let $X$ be a random vector distributed according to this measure. Then the sequence $|X_1|, |X_2|, \ldots, |X_n|$ is negatively associated.
\end{thm}

In particular we recover the negative association of absolute values for a random vector uniformly distributed on a generalized Orlicz ball.

\end{document}